\documentclass[12pt]{amsart}

\usepackage{amsmath}%
\usepackage{amsfonts}%
\usepackage{amssymb}%
\usepackage{graphicx}
\usepackage{setspace}

\title[$C^1$ Whitney Theorem for curves in $\mathbb{H}^n$]{The Whitney Extension Theorem for $C^1$, horizontal curves in the Heisenberg group}
\author{Scott Zimmerman}

\address{S.\ Zimmerman: Department of Mathematics, University of Pittsburgh, 301
  Thackeray Hall, Pittsburgh, PA 15260, USA, {\tt srz5@pitt.edu}}

\thanks{This work was supported by the NSF grant DMS-1500647.}
\keywords{Heisenberg group, Whitney extension theorem}
\subjclass[2010]{53C17}

plus 6pt minus 12pt
\newtheorem{theorem}{Theorem}
\newtheorem{lemma}[theorem]{Lemma}
\newtheorem{corollary}[theorem]{Corollary}
\newtheorem{proposition}[theorem]{Proposition}

\newtheorem{claim}[theorem]{Claim}
\theoremstyle{definition}
\newtheorem{remark}[theorem]{Remark}

\newcommand{\barint}{
\rule[.036in]{.12in}{.009in}\kern-.16in \displaystyle\int }

\newcommand{\barcal}{\mbox{$ \rule[.036in]{.11in}{.007in}\kern-.128in\int $}}
\def\mvint_#1{\mathchoice
          {\mathop{\vrule width 6pt height 3 pt depth -2.5pt
                  \kern -8pt \intop}\nolimits_{\kern -3pt #1}}%
          {\mathop{\vrule width 5pt height 3 pt depth -2.6pt
                  \kern -6pt \intop}\nolimits_{#1}}%
          {\mathop{\vrule width 5pt height 3 pt depth -2.6pt
                  \kern -6pt \intop}\nolimits_{#1}}%
          {\mathop{\vrule width 5pt height 3 pt depth -2.6pt
                  \kern -6pt \intop}\nolimits_{#1}}}

\numberwithin{theorem}{section} \numberwithin{equation}{section}

\begin{document}

\begin{abstract}
For a real valued function defined on a compact set $K \subset \mathbb{R}^m$, the classical Whitney Extension Theorem from 1934 gives necessary and sufficient conditions for the existence of a $C^k$ extension to $\mathbb{R}^m$.
In this paper, we prove a version of the Whitney Extension Theorem in the case of $C^1$, horizontal extensions for mappings defined on compact subsets of $\mathbb{R}$ taking values in the sub-Riemannian Heisenberg group $\mathbb{H}^n$.
\end{abstract}

\maketitle

\section{Introduction}

In 1934, Whitney \cite{whitney} discovered a necessary and sufficient condition for the existence of an extension $\tilde{f} \in C^k(\mathbb{R}^m)$ of a continuous function 
$f:K \to \mathbb{R}$ defined on a compact set $K \subset \mathbb{R}^m$.
The purpose of this paper is to prove a version of the Whitney Extension Theorem 
for mappings from a compact subset of $\mathbb{R}$ into the sub-Riemannian Heisenberg group $\mathbb{H}^n$.
See Section \ref{sec_heisenberg} for the definitions and properties of $\mathbb{H}^n$.
Applications of Whitney's extension theorem may be found in the construction of functions with unusual differentiability properties
(see \cite{whitney2})
and the existence of $C^1$ approximations for Lipschitz mappings (see \cite[Theorem 3.1.15]{federer} or Corollary \ref{corollary} below).
Such approximations are useful in the study of rectifiable sets, 
and the notion of rectifiability has seen recent activity in the setting of $\mathbb{H}^n$ (see for example \cite{balogh, fss, kirch, mattila}).
In fact, the authors in \cite{mattila} indicate that a Whitney type extension theorem into the Heisenberg group 
would help show the equivalence of two notions of rectifiability in $\mathbb{H}^n$.
For a comprehensive summary of the work done on Whitney type questions, see the introduction and references of \cite{fefferman}.

We say that a continuous function $f:K \to \mathbb{R}$ defined on a compact set $K \subset \mathbb{R}^m$
is of Whitney class $\mathfrak{C}^1(K)$ (equivalently $f \in \mathfrak{C}^1(K)$) if there is a continuous function $\mathcal{D} f \in C(K,\mathbb{R}^m)$ such that 
\begin{equation}
\label{whitney}
\lim_{\substack{|b-a| \to 0 \\ a,b \in K}} \frac{|f(b)-f(a)-\mathcal{D}f(a) \cdot (b-a)|}{|b-a|} = 0.
\end{equation}
We will call $\mathcal{D}f$ the \emph{derivative of f in the Whitney sense} or the \emph{Whitney derivative of} $f$.
Note that, a priori, $\mathcal{D}f$ is unrelated to the classical derivative since it is simply a continuous function defined on a compact set.

Condition \eqref{whitney} is necessary for the existence of a $C^1$ extension since any smooth function defined on $\mathbb{R}^m$ 
will satisfy \eqref{whitney} on a compact set $K \subset \mathbb{R}^m$ with Whitney derivative equal to the classical derivative.
Whitney proved that \eqref{whitney} is also sufficient to guarantee the existence of a $C^1$ extension. 
That is, for any compact $K \subset \mathbb{R}^m$ and $f \in \mathfrak{C}^1(K)$, there exists a function $\tilde{f} \in C^1(\mathbb{R}^m)$ such that 
$\tilde{f}|_K = f$ and $\nabla \tilde{f} | _K = \mathcal{D}f$.
See \cite{malgrange, whitney} for proofs of this.
Whitney actually proved a similar result with higher order regularity of $f$, but we will focus only on the first order case.

The Whitney class can be defined for mappings between higher dimensional Euclidean spaces in an obvious way.
A mapping $F: K \to \mathbb{R}^N$ 
is said to be of Whitney class $\mathfrak{C}^1(K, \mathbb{R}^N)$ (equivalently $F \in \mathfrak{C}^1(K, \mathbb{R}^N)$) for a compact $K \subset \mathbb{R}^m$ 
if each component $f_j$ of $F$ is of Whitney class $\mathfrak{C}^1(K)$ with Whitney derivative $\mathcal{D}f_j$.
Call $\mathcal{D}F = (\mathcal{D}f_1, \dots, \mathcal{D}f_N):K \to (\mathbb{R}^m)^N$ the Whitney derivative of $F$.
Given any $F \in \mathfrak{C}^1(K, \mathbb{R}^N)$, we may construct a $C^1$ extension of $F$ by applying Whitney's result to each of its components.

A natural question may be asked: what form would a sort of Whitney extension theorem take in the Heisenberg group?
In 2001, Franchi, Serapioni, and Serra Cassano \cite{fss} proved a $C^1$ version of the Whitney extension theorem for mappings from the Heisenberg group $\mathbb{H}^n$ into $\mathbb{R}$.
The authors provided a concise proof highlighting the major differences between the Euclidean and Heisenberg cases.
For a full exposition of the proof, see \cite{vittone}.
In this theorem, the function defined on a compact $K \subset \mathbb{H}^n$ is extended to $C_H^1$ function.
That is, the derivatives of the extension in the horizontal directions exist and are continuous.
In 2006, Vodop'yanov and Pupyshev \cite{vp} proved a $C^k$ version of Whitney's theorem for real valued functions defined on closed subsets of general Carnot groups.

In 2013, Piotr Haj\l{}asz posed the following two questions: 
\begin{itemize}
\item (Whitney extension) What are necessary and sufficient conditions for a continuous map $f: K \to \mathbb{R}^{2n+1}$ with $K \subset \mathbb{R}^m$ compact and $m \leq n$ to have a $C^1$ extension $\tilde{f} : \mathbb{R}^m \to \mathbb{R}^{2n+1}$ satisfying $\text{im} (D\tilde{f}(x)) \subset H_{\tilde{f}(x)}\mathbb{H}^n$ for every $x \in \mathbb{R}^m$?
\item ($C^1$ Luzin property) Is it true that, for every horizontal curve $\Gamma:[a,b] \to \mathbb{H}^n$ and any $\varepsilon>0$, there is a $C^1$, horizontal curve $\hat{\Gamma}:[a,b] \to \mathbb{H}^n$ such that
$$
| \{ s \in [a,b] \, | \, \hat{\Gamma}(s) \neq \Gamma(s) \} | < \varepsilon ?
$$
\end{itemize}

\begin{remark}
\normalfont
Note that the Whitney extension problem stated above is very different from the one solved by Franchi, Serapioni, and Serra Cassano since the nonlinear constraint now lies in the target space. 
Such a constraint makes the problem much more difficult.
\end{remark}

\begin{remark}
\normalfont
We only consider the Whitney problem in the case when $m \leq n$ since, if $m > n$, we have possible topological obstacles preventing the existence of a smooth extension.
For more details, see \cite{balogh,haj}.
\end{remark}

Let us consider the Whitney extension question in the case when $m = 1$.
For $K \subset \mathbb{R}$ compact, let $\Gamma = (f_1,g_1,\dots,f_n,g_n,h):K \to \mathbb{R}^{2n+1}$ 
be continuous so that there is a $C^1$, horizontal extension $\tilde{\Gamma}:\mathbb{R} \to \mathbb{R}^{2n+1}$.
Then clearly $\Gamma \in \mathfrak{C}^1(K,\mathbb{R}^{2n+1})$ with Whitney derivative $\Gamma' := \tilde{\Gamma}'|_K$.
That is,
\begin{equation}
\label{whitney2}
\lim_{\substack{|b-a| \to 0 \\ a,b \in K}} \frac{|\Gamma(b) - \Gamma(a) - (b-a)\Gamma'(a)|}{|b-a|} = 0.
\end{equation}
$\Gamma'$ must also satisfy the horizontality condition
\begin{equation}
\label{horizontal2}
h'(s)=2 \sum_{j=1}^n (f_j'(s) g_j(s) - f_j(s) g_j'(s))
\end{equation}
for any $s \in K$ (see \eqref{horizontal})
since any $C^1$, horizontal curve defined on $\mathbb{R}$ satisfies \eqref{horizontal2} for every $s \in \mathbb{R}$.
We may ask the following: are conditions \eqref{horizontal2} and \eqref{whitney2} sufficient to guarantee the existence of a horizontal, $C^1$ extension $\tilde{\Gamma}$ of $\Gamma$?
As we see here, the answer to this is, in general, ``no''.

\begin{proposition}
\label{counterexample}
There is a compact $K \subset \mathbb{R}$ and $\Gamma=(f,g,h) \in \mathfrak{C}^1(K,\mathbb{R}^3)$ with Whitney derivative $\Gamma'=(f',g',h')$ satisfying $h'=2 (f' g - f g')$ 
so that no $C^1$, horizontal curve $\tilde{\Gamma}: \mathbb{R} \to \mathbb{H}^1$ satisfies $\tilde{\Gamma}|_K = \Gamma$.
\end{proposition}

The next natural question to ask is the following: 
under what additional assumption does there exist a $C^1$, horizontal extension 
of $\Gamma \in \mathfrak{C}^1(K,\mathbb{R}^{2n+1})$?
The following proposition describes a necessary condition that every $C^1$, horizontal curve satisfies.

\begin{proposition}
\label{necessary}
Suppose $U \subset \mathbb{R}$ is open and $\Gamma = (f_1,g_1, \dots, f_n,g_n,h) : U \to \mathbb{H}^n$ is $C^1$ and horizontal. 
Then for any compact $K \subset U$
\begin{equation}
\label{necessaryeq}
\lim_{\substack{|b-a| \to 0 \\ a,b \in K}} \frac{\left|h(b)-h(a) - 2\sum_{j=1}^n (f_j(b)g_j(a) - f_j(a)g_j(b)) \right|}{|b-a|^2} = 0.
\end{equation}
\end{proposition}

The proofs of these two propositions are presented in Section~\ref{section_cc}.
As we will now see, the main result of this paper shows that assuming condition 
\eqref{necessaryeq} in addition to \eqref{whitney2} and \eqref{horizontal2} is in fact necessary and sufficient 
for the existence of a $C^1$, horizontal extension of a continuous $\Gamma: \mathbb{R} \supset K \to \mathbb{H}^n$.
This is summarized as follows:
\begin{theorem}
\label{main}
Suppose $K \subset \mathbb{R}$ is compact. 
Suppose $\Gamma = (f_1,g_1,\dots,f_n,g_n,h) : K \to \mathbb{H}^n$ 
is of Whitney class $\mathfrak{C}^1(K,\mathbb{R}^{2n+1})$ with Whitney derivative
$\Gamma' = (f_1',g_1', \dots, f_n',g_n',h')$.

Then there is a horizontal, $C^1$ curve $\tilde{\Gamma} : \mathbb{R} \to \mathbb{H}^n$ such that
$\tilde{\Gamma}|_K = \Gamma$ and $\tilde{\Gamma}'|_K = \Gamma'$
if and only if 
\begin{equation} \label{area}
\lim_{\substack{|b-a| \to 0 \\ a,b \in K}} \frac{\left| h(b)-h(a) - 2 \sum_{j=1}^n (f_j(b)g_j(a) - f_j(a)g_j(b)) \right|}{|b-a|^2} = 0
\end{equation}
and
\begin{equation}
\label{horiz}
h'(s)=2 \sum_{j=1}^n \left(f_j'(s)g_j(s)-g_j'(s)f_j(s) \right) \quad \text{for every } s \in K.
\end{equation}
\end{theorem}

\begin{remark}
\label{remark_}
\normalfont
We actually do not need to assume that $h \in \mathfrak{C}^1(K)$ 
because it is a consequence of \eqref{area} 
and the fact that $f_j \in \mathfrak{C}^1(K)$ and $g_j \in \mathfrak{C}^1(K)$ for $j = 1,\dots,n$.
The proof of this is simple, but it is contained at the end of Section~\ref{section_cc} for completeness.
\end{remark}

Theorem \ref{main} can be reformulated using the Lie group structure of $\mathbb{H}^n$ as follows:

\begin{theorem}
\label{main_pansu}
Suppose $K \subset \mathbb{R}$ is compact. 
Suppose $\Gamma = (f_1,g_1,\dots,f_n,g_n,h) : K \to \mathbb{H}^n$ 
and $\Gamma' = (f_1',g_1',\dots,f_n',g_n',h') : K \to \mathbb{H}^n$
are continuous.

Then there is a horizontal, $C^1$ curve $\tilde{\Gamma} : \mathbb{R} \to \mathbb{H}^n$ such that
$\tilde{\Gamma}|_K = \Gamma$ and $\tilde{\Gamma}'|_K = \Gamma'$
if and only if 
\begin{equation}
\label{pansu}
\lim_{\substack{b-a \to 0^+ \\ a,b \in K}} \left| \delta_{(b-a)^{-1}} \left( \Gamma(a)^{-1} * \Gamma(b) \right) - \Gamma_0'(a) \right| =0
\end{equation}
where $\Gamma_0' = (f_1',g_1',\dots,f_n',g_n',0)$, and 
$$
h'(s)=2 \sum_{j=1}^n \left(f_j'(s)g_j(s)-g_j'(s)f_j(s) \right) \quad \text{for every } s \in K.
$$
\end{theorem}

Here, $\delta_{(b-a)^{-1}}$ is the Heisenberg dilation defined at the end of Section \ref{sec_heisenberg}.
After assuming \eqref{horiz} and rewriting \eqref{pansu} using the definitions of the group law and dilations, 
we see that \eqref{pansu} is satisfied if and only if \eqref{area} is true and $\Gamma$ is of Whitney class $\mathfrak{C}^1(K,\mathbb{R}^{2n+1})$ with Whitney derivative $\Gamma'$.
That is, Theorem \ref{main} and Theorem \ref{main_pansu} are indeed equivalent.
Notice the similarity between the formulation of \eqref{pansu} and the definition of the Pansu derivative (see \cite{monti, pansu} for information on Pansu differentiation).
In fact, Proposition \ref{necessary} implies that $\Gamma_0'$ may be viewed as a Whitney-Pansu derivative of $\Gamma$.
Thus \eqref{pansu} acts as a sort of Whitney-Pansu condition for mappings defined on compact subsets of $\mathbb{R}$.

In 2015, Speight \cite{gareth} showed that a horizontal curve $\Gamma: [a,b] \to \mathbb{H}^n$ coincides with a $C^1$, horizontal curve $\hat{\Gamma}$ on $[a,b]$ up to a set of arbitrarily small measure.
That is, he answered the $C^1$ Luzin approximation question posed by Haj\l{}asz in the positive.
After seeing the paper by Speight, I quickly realized that this $C^1$ Luzin result follows from Theorem \ref{main}.
This is summarized at the end of this paper in Corollary~\ref{corollary}.
Moreover, Speight showed the surprising result that the Luzin approximation does \emph{not} hold for curves in the Engel group.

There are clear paths for future work on the subjects addressed in this paper. 
For example, a version of Theorem \ref{main} with $k > 1$ will be explained in a forthcoming paper with Gareth Speight \cite{speightandme}.

The paper is organized as follows.
In Section~\ref{sec_heisenberg}, we introduce the relevant geometric and analytic properties of the Heisenberg group.
In Section~\ref{section_cc}, we prove Propositions \ref{counterexample} and \ref{necessary} and Remark \ref{remark_}, and
Section~\ref{sec_proof} contains the proof of Theorem \ref{main} 
along with a new proof of Speight's result regarding the $C^1$ Luzin approximation for horizontal curves in $\mathbb{H}^n$.

The author would like to extend his sincerest gratitude to his advisor Dr. Piotr Haj\l{}asz for introducing him to the problem of Whitney extensions in the Heisenberg group and for his time and assistance proofreading this paper.
The author would also like to thank the referee for their helpful suggestions which led to an improvement of the paper.

\section{The Heisenberg group $\mathbb{H}^n$}
\label{sec_heisenberg}

The Heisenberg group $\mathbb{H}^n$ is $\mathbb{R}^{2n+1}$ given the structure of a Lie group with multiplication
\begin{align*}
(x_1,y_1, \dots, x_n,y_n,t)&*(x_1',y_1', \dots, x_n',y_n',t') \\
& = 
\Big(x_1+x_1', y_1+y_1', \dots , x_n + x_n' , y_n+y_n', t+t'+2 \sum_{j=1}^n (x_j' y_j - x_j y_j')\Big)
\end{align*}
with Lie algebra $\mathfrak{g}$ whose basis of left invariant vector fields is
$$
X_j(p) = \frac{\partial}{\partial x_j} + 2y_j \frac{\partial}{\partial t}, 
\quad 
Y_j(p)= \frac{\partial}{\partial y_j} - 2x_j \frac{\partial}{\partial t}, 
\quad 
T=\frac{\partial}{\partial t},
\quad
j=1,2,\ldots,n
$$
at any $p=(x_1, y_1, \dots, x_n , y_n,t) \in \mathbb{H}^n$.
We call $H \mathbb{H}^n = \mathrm{span} \{ X_1,Y_1,\dots, X_n, Y_n \}$ the {\em horizontal distribution} on $\mathbb{H}^n$, 
and denote by $H_p \mathbb{H}^n$ the horizontal space at $p$.
An absolutely continuous curve $\Gamma:[a,b] \to \mathbb{R}^{2n+1}$ is said to be {\em horizontal} 
if $\Gamma'(s) \in H_{\Gamma(s)} \mathbb{H}^n$ for almost every $s \in [a,b]$.
It is easy to see that the horizontal distribution is the kernel of the {\em standard contact form}
$$
\alpha = dt+2 \sum_{j=1}^n(x_j dy_j - y_j dx_j).
$$
That is, $H_p \mathbb{H}^n = \text{ker} \, \alpha (p)$. 
Hence it follows that an absolutely continuous curve $\Gamma = (f_1,g_1, \dots, f_n, g_n,h) = (\gamma,h)$ is horizontal if and only if 
\begin{equation}
\label{horizontal}
h'(s)=2 \sum_{j=1}^n (f_j'(s) g_j(s) - f_j(s) g_j'(s))
\end{equation}
for almost every $s \in [a,b]$.
This means that
\begin{align}
h(s)-h(a) 
&= 2 \sum_{j=1}^n \int_a^s (f_j'(\tau) g_j(\tau) - f_j(\tau) g_j'(\tau)) \, d \tau \label{area1}\\
&= 2 \sum_{j=1}^n \int_{a}^{s} \omega(\gamma_j'(\tau),\gamma_j(\tau)) \, d \tau \nonumber
\end{align}
for every $s \in [a,b]$ where $\gamma_j = (f_j,g_j)$ and $\omega$ is the \emph{standard symplectic form} on $\mathbb{R}^2$ defined as 
$$
\omega((u_1,u_2),(v_1,v_2)) = u_1v_2-u_2v_1 \quad \forall \, (u_1,u_2),(v_1,v_2) \in \mathbb{R}^2.
$$
Notice that, if $\Gamma$ is additionally assumed to be $C^1$ on $[a,b]$, then the continuity of the derivative implies $\Gamma'(s) \in H_{\Gamma(s)} \mathbb{H}^n$ for \emph{every} $s \in [a,b]$.
Recall that a continuous curve $\Gamma:[a,b] \to \mathbb{R}^{2n+1}$ is $C^1$ on $[a,b]$ if it is $C^1$ on $(a,b)$ 
and if $\Gamma'$ can be continuously extended to $[a,b]$.
Equivalently, $\Gamma$ is $C^1$ on $[a,b]$ if it can be extended to a $C^1$ curve defined on $\mathbb{R}$.

Suppose $\gamma = (f_1,g_1\dots, f_n, g_n):[a,b] \to \mathbb{R}^{2n}$ is absolutely continuous. 
If a value for $h(a)$ is fixed, then \eqref{area1} gives a unique horizontal curve $\Gamma=(\gamma,h):[a,b] \to \mathbb{H}^n$ whose projection onto the 
first $2n$ coordinates equals $\gamma$. 
We call this curve $\Gamma$ the {\em horizontal lift} of $\gamma$ with starting height $h(a)$.

Finally, the Heisenberg group has a natural family of dilations $\delta_r: \mathbb{H}^n \to \mathbb{H}^n$ 
defined for any $r>0$ by the equation
$$
\delta_r (x_1,y_1,\dots,x_n,y_n,t) = (rx_1,ry_1,\dots,rx_n,ry_n,r^2t).
$$

\section{Proofs of Propositions \ref{counterexample} and \ref{necessary}} \label{section_cc}

We will first prove Proposition \ref{necessary} as this result is used in the proof of Proposition~\ref{counterexample}.

\begin{proof}[Proof of Proposition \ref{necessary}]

Since $K$ is compact, we may assume without loss of generality that $U$ is bounded.
It suffices to prove \eqref{necessaryeq} when $U$ is an open interval.
Indeed, $U = \bigcup_{i=1}^{\infty} (a^i,b^i)$ for disjoint intervals $(a^i,b^i)$.
Since $K$ is compact, $K \subset \bigcup_{i=1}^N (a^i,b^i)$ for some $N \in \mathbb{N}$, 
and so we are only required to prove \eqref{necessaryeq} on each $(a^i,b^i) \cap K$ with $i \leq N$.
We may also replace $K$ by a possibly larger compact interval (also called $K$) contained in the interval $U$.

Since $\Gamma$ is horizontal, we have that $h'=2\sum_{j=1}^n (f_j'g_j-f_jg_j')$ on $U$. 
Choose $M>0$ so that $|f_j'|<M$ and $|g_j'|<M$ on $K$ for every $j=1,\dots,n$. %Choose $\delta>0$ small enough so that $\frac{|f_j(t) - f_j(a) - f_j'(a)(t-a)|}{|t-a|} < \varepsilon$ and $\frac{|g_j(t) - g_j(a) - g_j'(a)(t-a)|}{|t-a|} < \varepsilon$ for every $t,a \in K$ when $|t-a|<\delta$ and every $j=1,\dots, n$.
Fix $j \in \{ 1 , \dots, n \}$. %and $\varepsilon >0$.
For any $a,b \in K$ with $a<b$, %and $b-a<\delta$
we have $(a,b) \subset K$, and so

\begin{align*}
\int_a^{b} &f_j'(t)g_j(t) \, dt \\ 
&= \int_a^{b} f_j'(t)[g_j(a) + g_j'(a)(t-a) +g_j(t) - g_j(a) - g_j'(a)(t-a)] \, dt\\
&= g_j(a) \int_a^{b}f_j'(t) \, dt  + g_j'(a) \int_a^{b} f_j'(t) (t-a) \, dt + \int_a^{b} f_j'(t)[g_j(t) - g_j(a) - g_j'(a)(t-a)] \, dt.
\end{align*}
Now
$$
\frac{1}{(b-a)^2} \left| \int_a^{b} f_j'(t)[g_j(t) - g_j(a) - g_j'(a)(t-a)] \, dt \right| \leq \frac{M}{b-a} \int_a^{b} \frac{|g_j(t) - g_j(a) - g_j'(a)(t-a)|}{|t-a|} \, dt %<M \varepsilon
$$
which vanishes uniformly on $K$ as $|b-a| \to 0$ since $g_j$ is $C^1$.
In other words, $\int_a^{b} f_j'(t)[g_j(t) - g_j(a) - g_j'(a)(t-a)] \, dt = o(|b-a|^2)$ uniformly on $K$ as $|b-a| \to 0$.
Moreover
$$
\int_a^{b} f_j'(t) (t-a) \, dt = \int_a^{b} (f_j'(t)-f_j'(a)) (t-a) \, dt + \int_a^{b} f_j'(a)(t-a) \, dt.
$$
As above, we have $\int_a^{b} (f_j'(t)-f_j'(a)) (t-a) \, dt = o(|b-a|^2)$ uniformly on $K$ as $|b-a| \to 0$ since $f_j$ is $C^1$.
Thus we can write 
$$
\int_a^b f_j'(t)g_j(t) \, dt = [f_j(b)-f_j(a)]g_j(a) + g_j'(a)f_j'(a) \int_a^{b} (t-a) \, dt + o(|b-a|^2).
$$
Similar arguments yield
$$
\int_a^b g_j'(t)f_j(t) \, dt = [g_j(b)-g_j(a)]f_j(a) + f_j'(a)g_j'(a) \int_a^{b} (t-a) \, dt + o(|b-a|^2).
$$
Hence 
$$
\int_a^{b} (f_j'(t)g_j(t) - g_j'(t)f_j(t))\, dt =  f_j(b)g_j(a) - f_j(a)g_j(b) + o(|b-a|^2).
$$

Therefore,
\begin{align*}
h(b)-h(a) = \int_a^{b} h'(t) \, dt &= 2\sum_{j=1}^n \int_a^{b} (f_j'(t)g_j(t)-f_j(t)g_j'(t)) \, dt \\
&= 2\sum_{j=1}^n \left( f_j(b)g_j(a) - f_j(a)g_j(b) \right) + o(|b-a|^2)
\end{align*}
uniformly as $|b-a| \to 0$ for $a,b \in K$.
This completes the proof.
\end{proof}

As implied above, the following counterexample will fail to have a $C^1$, horizontal extension 
since any such extension would not satisfy the necessary condition outlined in Proposition \ref{necessary} on the compact set $K$.

\begin{proof}[Proof of Proposition \ref{counterexample}]
Let 
$$
K = \bigcup_{n=0}^{\infty} \left[ 1-\frac{1}{2^n} , 1 - \frac{3}{4} \cdot \frac{1}{2^{n}}  \right] \cup \{1\} =: \bigcup_{n=0}^{\infty}[c_n,d_n] \cup \{1\}.
$$
Then $K$ is compact. 
Define $\Gamma=(f,g,h):K \to \mathbb{H}^1$ so that, for each $n \in  \mathbb{N}\cup \{0\}$, 
$\Gamma(t) = (0,0,3^{-n})$ for $t \in [c_n,d_n]$, and set $\Gamma(1) = (0,0,0)$. 
Define $\Gamma'(t) = (0,0,0)$ for every $t \in K$.
We will show that 
\begin{equation}
\label{counterexamplegoal}
\frac{|\Gamma(b)-\Gamma(a) - (b-a)\Gamma'(a)|}{|b-a|}
\end{equation}
converges uniformly to 0 as $|b-a| \to 0$ on $K$.

Let $\varepsilon > 0$ and fix $n \in \mathbb{N}$ with $4(2/3)^n < \varepsilon$.
Suppose $a,b \in K$ with $|b-a|<2^{-(n+2)}$.
If $a$ and $b$ lie in the same interval $[c_k,d_k]$, then \eqref{counterexamplegoal} equals 0.
If $a$ and $b$ lie in different intervals $[c_k,d_k]$ and $[c_{\ell},d_{\ell}]$ for some $k,\ell \in \mathbb{N} \cup \{0\}$ (say $\ell > k$),
then $k \geq n$.
Indeed, if $k < n$, then 
$$
|b-a| \geq c_{\ell} - d_k \geq c_{k+1} - d_k = 2^{-(k+2)} > 2^{-(n+2)}
$$
which is impossible.
Hence, 
$$
\frac{|\Gamma(b)-\Gamma(a) - (b-a)\Gamma'(a)|}{|b-a|} 
\leq \frac{3^{-k}-3^{-{\ell}}}{c_{\ell}-d_k} 
\leq \frac{3^{-k}}{c_{k+1} - d_k}
= 4 \left( \frac{2}{3} \right)^k 
\leq 4 \left( \frac{2}{3} \right)^n
< \varepsilon.
$$
If either $a$ or $b$ equals 1 and the other point lies in the interval $[c_k,d_k]$ for some $k \in \mathbb{N} \cup \{0\}$, 
then, as in the above argument, $k \geq n$.
In this case, \eqref{counterexamplegoal} is bounded by $\frac{4}{3} \left( \frac{2}{3} \right)^n < \varepsilon$.
Thus $\Gamma \in \mathfrak{C}^1(K,\mathbb{R}^3)$, so there exists a $C^1$ extension of $\Gamma$ to all of $\mathbb{R}$.

Suppose now that $\tilde{\Gamma}=(\tilde{\gamma},\tilde{h}):\mathbb{R} \to \mathbb{H}^1$ is a $C^1$, horizontal extension of $\Gamma$. 
By Proposition~\ref{necessary}, 
$\tilde{\Gamma}$ must satisfy $|\tilde{h}(b)-\tilde{h}(a)|/|b-a|^2 \to 0$ uniformly on $K$ as $|b-a| \to 0$. 
However, $|c_{n+1} - d_n| = 2^{-(n+2)} \to 0$ as $n \to \infty$, but
$$
\frac{|h(c_{n+1})-h(d_n)|}{|c_{n+1}-d_n|^2} = \frac{3^{-n}-3^{-(n+1)}}{4^{-(n+2)} } = \frac{32}{3} \left( \frac{4}{3} \right)^n \to \infty
$$
as $n \to \infty$.
Thus $\Gamma$ has no $C^1$, horizontal extension to all of $\mathbb{R}$.
\end{proof}

We complete this section with the proof of Remark \ref{remark_}.

\begin{proof}
Let $K \subset \mathbb{R}$ be compact. 
Suppose $\Gamma = (f_1,g_1,\dots,f_n,g_n,h) : K \to \mathbb{H}^n$ 
satisfies \eqref{area}.
Suppose also that $f_j \in \mathfrak{C}^1(K)$ and $g_j \in \mathfrak{C}^1(K)$ for $j=1,\dots,n$
with Whitney derivatives $f_1',g_1', \dots, f_n',g_n'$.
Then
\begin{align*}
|h(b)-h(a)-(b-a)h'(a)| 
&\leq \left|h(b)-h(a)-2\sum_{j=1}^n(f_j(b)g_j(a) - f_j(a)g_j(b)) \right| \\
&\hspace{.3in} + 2\sum_{j=1}^n\big|f_j(b)g_j(a) - f_j(a)g_j(b) \\
&\hspace{1in} - (b-a)\left(f_j'(a)g_j(a) - f_j(a)g_j'(a) \right) \big| \\
&= o(|b-a|^2) + 2\sum_{j=1}^n \big|g_j(a) \left( f_j(b) - (b-a)f_j'(a) \right) \\
&\hspace{1.6in} - f_j(a) \left(g_j(b) - (b-a)g_j'(a) \right) \big| \\
&= o(|b-a|^2) + o(|b-a|)
\end{align*}
uniformly as $|b-a| \to 0$ for any $a,b \in K$.
That is, $h \in \mathfrak{C}^1(K)$.
\end{proof}

\section{Proof of the main result and the Luzin property}
\label{sec_proof}

\begin{proof} [Proof of Theorem \ref{main}]
Write $\Gamma = (\gamma, h) = (\gamma_1, \dots, \gamma_n, h)$ where $\gamma_j = (f_j,g_j):K \to \mathbb{R}^2$.

The necessity of conditions \eqref{area} and \eqref{horiz} was verified in Proposition \ref{necessary} 
and in the discussion preceding Proposition \ref{counterexample}.
We will now prove that these are sufficient conditions.

Since $K$ is compact, we can define the closed interval $I=[\min \{ K \}, \max \{ K \}]$. 
Thus $I \setminus K$ is open, so $I \setminus K = \bigcup_{i} (a^i,b^i)$ 
for pairwise disjoint open intervals $(a^i,b^i)$. 
To construct the extension $\tilde{\Gamma}$ of $\Gamma$, 
we will define a $C^1$ extension $\tilde{\gamma}$ of $\gamma$ on each interval $[a^i,b^i]$ 
so that the horizontal lift of $\tilde{\gamma}$ will coincide with $\Gamma$ on $K$.

If the collection $\{(a^i, b^i)\}_{i}$ is finite, then we can construct the extension directly. 
On each $[a^i,b^i]$ and for each $j \in \{ 1 ,\dots, n \}$ define $\tilde{\gamma}_j^i=(\tilde{f}_j^i,\tilde{g}_j^i):[a^i,b^i] \to \mathbb{R}^2$ to be a curve which is $C^1$ on $[a^i,b^i]$ satisfying 
\begin{equation}
\label{i}
\tilde{\gamma}_j^i(a^i) = \gamma_j(a^i) 
\quad \text{ and } \quad 
\tilde{\gamma}_j^i(b^i) = \gamma_j(b^i),
\end{equation}
\begin{equation}
\label{ii}
(\tilde{\gamma}_j^i)'(a^i) = \gamma_j'(a^i) 
\quad \text{ and } \quad 
(\tilde{\gamma}_j^i)'(b^i) = \gamma_j'(b^i),
\end{equation}
\begin{equation}
\label{iii}
2 \int_{a^i}^{b^i} \left( (\tilde{f}_j^i)' \tilde{g}_j^i - \tilde{f}_j^i (\tilde{g}_j^i)' \right) 
= \frac{1}{n} \left[ h(b^i)-h(a^i) \right].
\end{equation}

The fact that a curve exists satisfying the first two conditions is obvious.
The value on the right hand side of condition \eqref{iii} is fixed, 
and the integral on the left may be interpreted as an area via Green's theorem.
Thus any curve with prescribed values at $a^i$ and $b^i$ as in \eqref{i} and \eqref{ii} 
can be adjusted in $(a^i,b^i)$ without disturbing the curve at the endpoints
so that this integral condition \eqref{iii} is indeed satisfied.

Now define the curve $\tilde{\gamma}:I \to \mathbb{R}^{2n}$ so that 
$$
\tilde{\gamma}|_K = \gamma 
\quad \text{ and } \quad
\tilde{\gamma}|_{(a^i,b^i)} = (\tilde{\gamma}_1^i, \dots, \tilde{\gamma}_n^i)
$$
for every $i \in \mathbb{N}$.
The properties \eqref{i} and \eqref{ii} above ensure that $\tilde{\gamma}$ is $C^1$ on $I$.
Extend $\tilde{\gamma}$ to a $C^1$ curve on all of $\mathbb{R}$.
Finally, define $\tilde{\Gamma}=(\tilde{\gamma},\tilde{h})$ to be the unique horizontal lift of 
$\tilde{\gamma}$ so that $\tilde{h}(\min \{ K \}) = h(\min \{ K \})$.
Property \eqref{iii} ensures that this lift is a $C^1$ extension of $\Gamma$ since, on $[a^i,b^i]$, 
the horizontal lift traverses the distance $h(b^i) - h(a^i)$ in the vertical direction (as described in \eqref{area1}).

Now, consider the case when the collection $\{(a^i, b^i)\}$ is infinite.
The simple construction above can not in general be applied directly in this case.
Indeed, in the above construction, there was little control on the behavior of the curves.
For example, curves filling a small gap from $\gamma_j(a^i)$ to $\gamma_j(b^i)$ could be made arbitrarily long.
Thus we must now be more careful when constructing these curves.

Notice that the sequence $\{(a^i, b^i)\}_{i=1}^{\infty}$ satisfies $b^i - a^i \to 0$ as $i \to \infty$ since $I$ is bounded and the intervals are disjoint.
Thus, using the fact that each $f_j \in \mathfrak{C}^1(K)$ and $g_j \in \mathfrak{C}^1(K)$ and using \eqref{area}, we can find a non-increasing sequence $\varepsilon^i \to 0$ so that the following conditions hold for each $i \in \mathbb{N}$:
$$ 
b^i-a^i < \varepsilon^i, 
\hspace{.4in}
|\gamma(b^i)-\gamma(a^i)| < \varepsilon^i,
$$

$$
\left| \frac{\gamma(b^i)-\gamma(a^i)-(b^i-a^i)\gamma'(a^i)}{b^i-a^i} \right| < \varepsilon^i,
\hspace{.5in}
\left| \frac{\gamma(b^i)-\gamma(a^i)-(b^i-a^i)\gamma'(b^i)}{b^i-a^i} \right| < \varepsilon^i,
$$

$$
\frac{1}{n} \left| \frac{ h(b^i)-h(a^i) - \sum_{j=1}^n(f_j(b^i)g_j(a^i) - f_j(a^i)g_j(b^i))}{(b^i-a^i)^2} \right| < \varepsilon^i.
$$

Our plan for the proof will be as follows: 
for each $i \in \mathbb{N}$ we will construct a horizontal curve $\tilde{\Gamma}^i$ in $\mathbb{H}^n$ defined on $[a^i,b^i]$
connecting $\Gamma(a^i)$ to $\Gamma(b^i)$ and satisfying conditions \eqref{i}, \eqref{ii}, and \eqref{iii}.
In addition, the curves will be constructed in a controlled way so that the concatenation of all of these curves 
creates a $C^1$, horizontal extension of $\Gamma$.
To create these curves in $\mathbb{H}^n$, we will first define for each $i \in \mathbb{N}$ curves 
$\tilde{\gamma}_j^i:[a^i,b^i] \to \mathbb{R}^2$ in each $x_j y_j$-plane
connecting $\gamma_j(a^i)$ to $\gamma_j(b^i)$.
Horizontally lifting each curve $\tilde{\gamma}^i = (\tilde{\gamma}_1^i, \dots, \tilde{\gamma}_n^i)$
to $\tilde{\Gamma}^i$ will create an extension $\tilde{\Gamma}:I \to \mathbb{H}^n$ of $\Gamma$.
The controlled construction of the curves $\tilde{\gamma}_j^i$ together with \eqref{iii} will ensure that 
this extension $\tilde{\Gamma}$ is indeed $C^1$.

%With enough control over the curves $\tilde{\gamma}_j^i$, extending the $\gamma_j$ mappings from $K$ via $\tilde{\gamma}_j^i$ will give us a $C^1$ curve $\tilde{\gamma}=(\tilde{\gamma}_1, \dots, \tilde{\gamma}_n):I \to \mathbb{R}^{2n}$. Clearly the horizontal lift of $\tilde{\gamma}$ is $C^1$. However, condition \eqref{iii} will ensure that the horizontal lift $\tilde{\Gamma}$ of $\tilde{\gamma}$ is indeed a $C^1$ extension of $\Gamma$. That is, $\tilde{\Gamma}$ coincides with $\Gamma$ on $K$.

We begin with the following lemma in which we define a curve $\eta_j^i$ from $[a^i,b^i]$ into the $x_jy_j$-plane 
sending $a^i$ to the origin and $b^i$ to $\left( |\gamma_j(b^i)-\gamma_j(a^i)|,0 \right)$
Later, we will compose the curves in the lemma with planar rotations and translations to create the curves $\tilde{\gamma}_j^i$ connecting $\gamma_j(a^i)$ to $\gamma_j(b^i)$ as described above.

We now introduce some notation.
Fix $j \in \{1 , \dots , n \}$.
For each $i \in \mathbb{N}$, if $|\gamma_j(b^i)-\gamma_j(a^i)| > 0$, let $\mathbf{u}_j^i=\frac{\gamma_j(b^i)-\gamma_j(a^i)}{|\gamma_j(b^i)-\gamma_j(a^i)|}$ and let $\mathbf{v}_j^i$ be the unit vector perpendicular to $\mathbf{u}_j^i$ given by a counter-clockwise rotation of $\mathbf{u}_j^i$ in the $x_jy_j$-plane.
If $|\gamma_j(b^i)-\gamma_j(a^i)| = 0$, define $\mathbf{u}_j^i$ and $\mathbf{v}_j^i$ to be the unit vectors pointing in the $x_j$ and $y_j$ coordinate directions respectively.
Since each $\gamma_j$ is of Whitney class $\mathfrak{C}^1(K,\mathbb{R}^2)$, we may choose $M>0$ so that $\frac{|\gamma_j(b^i)-\gamma_j(a^i)|}{|b^i - a^i|} < M$, $|\gamma_j'(a^i)|<M$, and $|\gamma_j'(b^i)| <M$ for every $i \in \mathbb{N}$ and every $j=1,\dots,n$.

\begin{lemma}
Fix $i \in \mathbb{N}$ and $j \in \{ 1 ,\dots n \}$.
There exists a $C^1$ curve $\eta_j^i=(x_j^i,y_j^i):[a^i,b^i] \to \mathbb{R}^2$ satisfying
\begin{equation}
\label{cond_point}
\eta_j^i(a^i) = (0,0) 
\quad \text{ and } \quad 
\eta_j^i(b^i) = \left( |\gamma_j(b^i) - \gamma_j(a^i)| , 0 \right),
\end{equation}
\begin{equation}
\label{cond_deriv}
(\eta_j^i)'(a^i) = (\gamma_j'(a^i) \cdot \mathbf{u}_j^i, \gamma_j'(a^i) \cdot \mathbf{v}_j^i) 
\quad \text{ and } \quad 
(\eta_j^i)'(b^i) = (\gamma_j'(b^i) \cdot \mathbf{u}_j^i,\gamma_j'(b^i) \cdot \mathbf{v}_j^i),
\end{equation}
\begin{equation}
\label{cond_beta}
||\eta_j^i||_{\infty} < P(\varepsilon^i)
\quad \text{ and } \quad 
||(\eta_j^i)' - (\gamma_j'(a^i) \cdot \mathbf{u}_j^i, \gamma_j'(a^i) \cdot \mathbf{v}_j^i)||_{\infty} < P(\varepsilon^i)
\end{equation}
where $P(t) = C'( t^{1/2} %+ t + t^{3/2} 
+ t^2)$ for every $t \geq 0$ and some constant $C' \geq 0$ depending only on $M$, and
\begin{equation}
\label{cond_area}
2 \int_{a^i}^{b^i} \left((x_j^i)' y_j^i - x_j^i (y_j^i)' \right) = \frac{1}{n} \left[ h(b^i)-h(a^i) - 2\sum_{m=1}^n \left(f_m(b^i)g_m(a^i) - f_m(a^i)g_m(b^i) \right) \right].
\end{equation}
\end{lemma}
The proof of this lemma is omitted here for continuity. It is presented in the appendix.

\begin{remark}
\label{remark}
\normalfont
Observe that $\gamma_j'(a^i) \cdot \mathbf{u}_j^i$ and $\gamma_j'(a^i) \cdot \mathbf{v}_j^i$ are the components of the vector
$\gamma_j'(a^i)$ in the orthonormal basis $\langle \mathbf{u}_j^i, \mathbf{v}_j^i \rangle$.
Soon, we will define the curve $\tilde{\gamma}_j^i$ by moving $\eta_j^i$ via a rotation and translation.
The rotation will map the standard basis in the $x_jy_j$-plane to $\langle \mathbf{u}_j^i, \mathbf{v}_j^i \rangle$,
and hence \eqref{cond_deriv} will imply 
$(\tilde{\gamma}_j^i)'(a^i) = \gamma_j'(a^i)$ and $(\tilde{\gamma}_j^i)'(b^i) = \gamma_j'(b^i)$.
The translation will map the segment connecting the origin and $(|\gamma_j(b^i)-\gamma_j(a^i)|,0)$ 
to the segment $\overline{\gamma_j(a^i)\gamma_j(b^i)}$, 
and so \eqref{cond_point} will give 
$\tilde{\gamma}_j^i(a^i) = \gamma_j(a^i)$ and $\tilde{\gamma}_j^i(b^i) = \gamma_j(b^i)$.
Condition \eqref{cond_beta} exhibits control on the $C^1$ norm of $\eta_j^i$ 
and will thus give us control on the $C^1$ norm of its isometric image $\tilde{\gamma}_j^i$.
Note also that the integral condition \eqref{cond_area} seems more complicated than \eqref{iii}.
However, after rotating and translating $\eta_j^i$, \eqref{cond_area} will reduce to \eqref{iii}.
\end{remark}

Fix $j \in \{1 \dots, n \}$ and $i \in \mathbb{N}$.
Define the curve $\eta_j^i:[a^i,b^i] \to \mathbb{R}^2$ as in the lemma.
Define the isometry $\Phi_j^i: \mathbb{R}^{2} \to \mathbb{R}^{2}$ as $\Phi_j^i(p) = A_j^i p + c_j^i$ for $p \in \mathbb{R}^2$
where $c_j^i = (f_j(a^i), g_j(a^i))$ and
$$
A_j^i = \frac{1}{|\gamma_j(b^i)-\gamma_j(a^i)|}
\left( \begin{array}{cc}
f_j(b^i)-f_j(a^i) & -(g_j(b^i)-g_j(a^i))\\
g_j(b^i)-g_j(a^i) & f_j(b^i)-f_j(a^i)\\ \end{array} \right)
$$
when $|\gamma_j(b^i)-\gamma_j(a^i)| \neq 0$ and $A_j^i = I_{2 \times 2}$ if $|\gamma_j(b^i)-\gamma_j(a^i)|=0$.
That is, $\Phi_j^i$ is the isometry described in Remark \ref{remark}.
When $\gamma_j(a^i)=\gamma_j(b^i)$, $\Phi_j^i$ is simply a translation sending the origin to $\gamma_j(a^i)$ without any rotation.
Now define $\tilde{\gamma}_j^i = \Phi_j^i \circ \eta_j^i: [a^i,b^i] \to \mathbb{R}^2$.
Hence $\tilde{\gamma}_j^i$ is a $C^1$ curve in $\mathbb{R}^2$ connecting $\gamma_j(a^i)$ to $\gamma_j(b^i)$.

Write $\tilde{\gamma}^i = (\tilde{\gamma}_1^i, \dots, \tilde{\gamma}_n^i): [a^i,b^i] \to \mathbb{R}^{2n}$.
Now, define $\tilde{\Gamma}^i:[a^i,b^i] \to \mathbb{H}^n$ to be the unique horizontal lift of $\tilde{\gamma}^i$ 
with starting height $h(a^i)$.
The resulting lift is $C^1$ on $[a^i,b^i]$ by definition.
Define $\tilde{\Gamma}:I \to \mathbb{H}^n$ so that 
$$
\tilde{\Gamma}|_K = \Gamma \quad \text{ and } \quad \tilde{\Gamma}|_{(a^i,b^i)} = \tilde{\Gamma}^i
$$ 
for each $i \in \mathbb{N}$.
Write $\tilde{\Gamma} = (\tilde{\gamma},\tilde{h}) = (\tilde{\gamma}_1, \dots, \tilde{\gamma}_n, \tilde{h})$ where 
$\tilde{\gamma}_j = (\tilde{f}_j, \tilde{g}_j)$ for each $j \in \{ 1 ,\dots , n \}$.
It remains to show that $\tilde{\Gamma}$ is $C^1$ on all of $I$.
Notice that we do not yet know if $\tilde{\Gamma}$ is even continuous.

\begin{claim}
\label{integral}
For each $i \in \mathbb{N}$, $\int_{a^i}^{b^i} \tilde{h}' = h(b^i) - h(a^i)$.
\end{claim}
Fix $j \in \{ 1, \dots, n \}$ and suppose $|\gamma_j(b^i) - \gamma_j(a^i)| \neq 0$.
We have
$$
(\Phi_j^i \circ \eta_j^i) =
\left( \begin{array}{c}
x_j^i \frac{f_j(b^i)-f_j(a^i)}{|\gamma_j(b^i) - \gamma_j(a^i)|} 
- y_j^i \frac{g_j(b^i)-g_j(a^i)}{|\gamma_j(b^i) - \gamma_j(a^i)|} + f_j(a^i)\\
x_j^i \frac{g_j(b^i)-g_j(a^i)}{|\gamma_j(b^i) - \gamma_j(a^i)|} 
+ y_j^i \frac{f_j(b^i)-f_j(a^i)}{|\gamma_j(b^i) - \gamma_j(a^i)|} + g_j(a^i)
\end{array} \right)
$$
where $\eta_j^i = (x_j^i, y_j^i)$. This gives
\begin{align*}
\omega(\tilde{\gamma}_j',\tilde{\gamma}_j) &= \omega((\Phi_j^i \circ \eta_j^i)',(\Phi_j^i \circ \eta_j^i)) \\
&= g_j(a^i) \left((x_j^i)' \frac{f_j(b^i)-f_j(a^i)}{|\gamma_j(b^i) - \gamma_j(a^i)|} - (y_j^i)' \frac{g_j(b^i)-g_j(a^i)}{|\gamma_j(b^i) - \gamma_j(a^i)|} \right) \\
&-f_j(a^i) \left((x_j^i)' \frac{g_j(b^i)-g_j(a^i)}{|\gamma_j(b^i) - \gamma_j(a^i)|} + (y_j^i)' \frac{f_j(b^i)-f_j(a^i)}{|\gamma_j(b^i) - \gamma_j(a^i)|} \right) \\
&+\left( \left(\frac{f_j(b^i)-f_j(a^i)}{|\gamma_j(b^i) - \gamma_j(a^i)|} \right)^2 + \left( \frac{g_j(b^i)-g_j(a^i)}{|\gamma_j(b^i) - \gamma_j(a^i)|} \right)^2 \right) ((x_j^i)' y_j^i - x_j^i (y_j^i)').
\end{align*}
Now since the constructions in the lemma give 
$$
\int_{a^i}^{b^i} (y_j^i)' = y_j^i(b^i)-y_j^i(a^i) = 0 \text{ and } \int_{a^i}^{b^i} (x_j^i)' = x_j^i(b^i)-x_j^i(a^i) = |\gamma_j(b^i) - \gamma_j(a^i)|
$$
and
$$
\left( \frac{f_j(b^i)-f_j(a^i)}{|\gamma_j(b^i) - \gamma_j(a^i)|} \right)^2 + \left( \frac{g_j(b^i)-g_j(a^i)}{|\gamma_j(b^i) - \gamma_j(a^i)|}\right)^2 = \left( \frac{|\gamma_j(b^i) - \gamma_j(a^i)|}{|\gamma_j(b^i) - \gamma_j(a^i)|}\right)^2 = 1,
$$
we have 
$$
2\int_{a^i}^{b^i} \omega((\Phi_j^i \circ \eta_j^i)',(\Phi_j^i \circ \eta_j^i)) 
= 2(f_j(b^i)g_j(a^i)-f_j(a^i)g_j(b^i)) +  2\int_{a^i}^{b^i} \left((x_j^i)' y_j^i - x_j^i (y_j^i)' \right).
$$
By condition (\ref{cond_area}), we have
$$
2 \int_{a^i}^{b^i} \left((x_j^i)' y_j^i - x_j^i (y_j^i)' \right) = \frac{1}{n} \left[ h(b^i)-h(a^i) - 2\sum_{m=1}^n (f_{m}(b^i)g_{m}(a^i) - f_{m}(a^i)g_{m}(b^i)) \right],
$$
thus
\begin{align*}
\int_{a^i}^{b^i} \tilde{h}' &= 2 \sum_{j=1}^n \int_{a^i}^{b^i} \omega((\Phi_j^i \circ \eta_j^i)',(\Phi_j^i \circ \eta_j^i)) \\
&= \sum_{j=1}^n \Bigg[ 2 \left( f_j(b^i)g_j(a^i)-f_j(a^i)g_j(b^i) \right) \\
& \hspace{.7in} + \frac{1}{n} \Bigg[ h(b^i)-h(a^i)  - 2\sum_{{m}=1}^n (f_{m}(b^i)g_{m}(a^i) - f_{m}(a^i)g_{m}(b^i)) \Bigg] \Bigg] \\
&= h(b^i) - h(a^i).
\end{align*}

If $|\gamma_j(b^i) - \gamma_j(a^i)| = 0$, then $A^i$ is the identity.
Thus 
\begin{align*}
\int_{a^i}^{b^i} \tilde{h}' 
= 2 \sum_{j=1}^{n} \int_{a^i}^{b^i} \left[ (x_j^i)'(y_j^i + g_j(a^i)) - (y_j^i)'(x_j^i + f_j(a^i)) \right]
&= 2 \sum_{j=1}^{n} \int_{a^i}^{b^i} \left( (x_j^i)'y_j^i - (y_j^i)'x_j^i \right) \\
&= h(b^i) - h(a^i)
\end{align*}
since $f_j(a^i) = f_j(b^i)$ and $g_j(a^i) = g_j(b^i)$ in this case.
This completes the proof of the claim.

\begin{claim}
\label{cont}
$\sup_{s \in [a^i,b^i]} |\tilde{\gamma}(s) - \gamma(a^i)| \to 0$ as $i \to \infty$.
\end{claim}
We have for any $i \in \mathbb{N}$, any $s \in [a^i,b^i]$, and any $j \in \{1,\dots,n \}$,
$$
|\tilde{\gamma}_j(s) - \gamma_j(a^i)| 
= |\Phi_j^i( \eta_j^i(s)) - \gamma_j(a^i)|
= | \eta_j^i(s)| < P(\varepsilon^i)
$$
by \eqref{cond_beta} since $(\Phi_j^i)^{-1}(\gamma_j(a^i)) = (0,0)$ and $\Phi_j^i$ is an isometry on $\mathbb{R}^2$.
Since $P(\varepsilon^i) \to 0$ as $i \to \infty$, the claim is proven.

\begin{claim}
\label{deriv_claim}
$\sup_{s \in [a^i,b^i]} |\tilde{\Gamma}'(s) - \Gamma'(a^i)| \to 0$ as $i \to \infty$.
\end{claim}
We have for any $i \in \mathbb{N}$, any $s \in [a^i,b^i]$, and any $j \in \{1,\dots,n \}$,
\begin{align*}
|\tilde{\gamma}_j'(s) - \gamma_j'(a^i)| 
&= |(\Phi_j^i \circ \eta_j^i)'(s)) - \gamma_j'(a^i)| \\
&= |A_j^i ((\eta_j^i)'(s)) - A_j^i (\gamma_j'(a^i) \cdot \mathbf{u}_j^i, \gamma_j'(a^i) \cdot \mathbf{v}_j^i)| \\
&= |(\eta_j^i)'(s) - (\gamma_j'(a^i) \cdot \mathbf{u}_j^i, \gamma_j'(a^i) \cdot \mathbf{v}_j^i)| < P(\varepsilon^i)
\end{align*}
by \eqref{cond_beta}. Finally, by \eqref{horiz} and the definition of a horizontal lift

%Suppose $\gamma_j'(t) \neq 0$.
%By the previous claim, there is some $N>0$ so that 
%$$
%\left|(\gamma_j'(a^i) + \gamma_j'(b^i)) \cdot \mathbf{u}_j^i - 9 \frac{|\gamma_j(b^i)-\gamma_j(a^i)|}{|b^i-a^i|} \right| > \sqrt{\varepsilon^i}
%$$
%for every $i>N$.
%It follows that $\eta_j^i$ is defined on these $(a^i,b^i)$ as described in Lemma~1 for large enough $i$.
%Thus, by the lemma construction, for any $s \in (a^i,b^i)$,
%$$
%|(\eta_j^i)'(s) - (\gamma_j'(a^i) \cdot \mathbf{u}_j^i, \gamma_j'(a^i) \cdot \mathbf{v}_j^i)| 
%< \sqrt{2} P(\varepsilon^i).
%$$ 

%Suppose now that $\gamma_j'(t) = 0$.
%For any $i \in \mathbb{N}$, if $\eta_j^i$ is defined on $(a^i,b^i)$ as described in Lemma~1, 
%then we have for any $s \in (a^i,b^i)$, $|\tilde{\gamma}'(s) - \gamma'(a^i)| < \sqrt{2} P(\varepsilon^i)$.
%If $\eta_j^i$ is defined on $(a^i,b^i)$ as described in Lemma 2, 
%then for any $s \in (a^i,b^i)$
%$$
%|\tilde{\gamma}_j'(s) - \gamma_j'(a^i)| 
%\leq |(\eta_j^i)'(s)| + |\gamma_j'(a^i)| 
%< P(\varepsilon^i) + |\gamma_j'(a^i)|,
%$$
%and $|\gamma_j'(a^i)| \to |\gamma_j'(t)| = 0$ as $i \to \infty$.
%Since $P(\varepsilon^i)$ vanishes as $i \to \infty$, 
%we conclude that $\sup_{s \in (a^i,b^i)} |\tilde{\gamma}'(s) - \gamma'(a^i)| \to 0$ as $i \to \infty$ regardless of which lemma was used to define $\eta_j^i$.

\begin{align*}
\sup_{s \in [a^i,b^i]} &|\tilde{h}'(s) - h'(a^i)| \\
&\leq \sup_{s \in [a^i,b^i]} 2 \sum_{j=1}^n \left|(\tilde{f}_j'(s) \tilde{g}_j(s) - \tilde{g}_j'(s) \tilde{f}_j(s)) - (f_j'(a^i) g_j(a^i) - g_j'(a^i) f_j(a^i)) \right| 
\end{align*}
which can be made arbitrarily small as $i \to \infty$ because of the convergences in Claim~\ref{cont} and this claim.
This proves the claim.

By definition, $\tilde{\Gamma}$ is $C^1$ on $(a^i,b^i)$ for any $i \in \mathbb{N}$, so it is $C^1$ on $I \setminus K$.
We will now verify the differentiability of $\tilde{\Gamma}$ on $K$.

\begin{claim}
For any $t \in K$, $\tilde{\Gamma}$ is differentiable at $t$ and $\tilde{\Gamma}'(t) = \Gamma'(t)$.
\end{claim}
Suppose $t \in K$ (so $\tilde{\Gamma}(t) = \Gamma(t)$).
If $t = a^i$ for some $i \in \mathbb{N}$, then for any $j \in \{1,\dots,n \}$ and $0 < \delta < b^i-a^i$, 
we can use the definition of $\tilde{\gamma}^i$ to write
\begin{align*}
\delta^{-1} |\tilde{\gamma}_j(a^i + \delta) &- \tilde{\gamma}_j(a^i) - \delta \gamma_j'(a^i)|\\
&= \delta^{-1} |A_j^i(\eta_j^i(a^i+\delta)) + c_j^i - (A_j^i(\eta_j^i(a^i)) + c_j^i) - \delta A_j^i((\eta_j^i)'(a^i))|\\
&= \delta^{-1} |\eta_j^i(a^i+\delta) - \eta_j^i(a^i) - \delta (\eta_j^i)'(a^i)|
\end{align*}
which vanishes as $\delta \to 0$ since $\eta_j^i$ is differentiable from the right at $a^i$.
Thus $\tilde{\gamma}$ is differentiable from the right at $a^i$ and the right derivative equals $\gamma'(a^i)$.
Moreover,
$$
\lim_{\delta \to 0^+} \tilde{\gamma}_j'(a^i + \delta) 
= \lim_{\delta \to 0^+} A_j^i((\eta_j^i)'(a^i + \delta))
= A_j^i((\eta_j^i)'(a^i))
= \gamma_j'(a^i).
$$
Thus $\tilde{\gamma}'$ is continuous from the right at $a^i$.
Now $\tilde{\Gamma}$ was constructed on $(a^i,b^i)$ by lifting $\gamma(a^i)$ to the height $h(a^i)$.
Hence $\int_{a^i}^c \tilde{h}' = \tilde{h}(c) - h(a^i)$ for any $c \in (a^i,b^i)$.
Thus for $0 < \delta < b^i-a^i$,
\begin{align*}
\delta^{-1} |\tilde{h}&(a^i + \delta) - \tilde{h}(a^i) - \delta h'(a^i)|\\
&\leq 2 \delta^{-1} \sum_{j=1}^n \int_{a^i}^{a^i + \delta} \left|(\tilde{f}_j'(s) \tilde{g}_j(s) - \tilde{g}_j'(s) \tilde{f}_j(s)) - (f_j'(a^i) g_j(a^i) - g_j'(a^i) f_j(a^i)) \right| \, ds
\end{align*}
which vanishes as $\delta \to 0$ by the right sided continuity of $\tilde{\gamma}$ and $\tilde{\gamma}'$ at $a^i$.
Therefore $\tilde{\Gamma}$ is differentiable from the right at $a^i$ and the right derivative is $\Gamma'(a^i)$.

We can similarly argue to show that $\tilde{\gamma}$ is differentiable from the left at $b^i$ for any $i \in \mathbb{N}$
with left derivative equal to $\gamma'(b^i)$ and that $\tilde{\gamma}'$ is continuous from the left at $b^i$.
Applying Claim~\ref{integral} with $0 < \delta < b^i-a^i$ gives
\begin{align*}
\delta^{-1} |\tilde{h}(b^i - \delta) - \tilde{h}(b^i) + \delta h'(b^i)|
& =\delta^{-1} |\delta h'(b^i) + (\tilde{h}(b^i - \delta) -\tilde{h}(a^i)) - (h(b^i)-h(a^i))|\\
&\leq \delta^{-1} \int_{b^i - \delta}^{b^i} \left|h'(b^i) - \tilde{h}'(s) \right| \, ds
\end{align*}
which vanishes as $\delta \to 0$ as above.
Therefore $\tilde{\Gamma}$ is differentiable from the left $b^i$ and the left derivative equals $\Gamma'(b^i)$.

We will now show that $\tilde{\Gamma}$ is differentiable from the right at any $t \in K$.
Suppose now that $t \neq a^i$ for any $i \in \mathbb{N}$ since we already proved right hand differentiability at $a^i$ above.
(We may also suppose that $t \neq \max \{ K \}$.)
Fix $\tilde{\varepsilon}>0$.
Let $\{ t^k \}$ be any decreasing sequence in $K$ with $t^k \to t$. 
Since $\Gamma \in \mathfrak{C}^1(K,\mathbb{R}^{2n+1})$, 
there is some $N>0$ so that for any $k>N$
$$
(t^k-t)^{-1} |\Gamma(t^k) - \Gamma(t) - (t^k-t) \Gamma'(t) | < \tilde{\varepsilon}.
$$

Suppose there is a decreasing sequence $\{ t^k \}$ in $I \setminus K$ with $t^k \to t$.
Then $t^k \in (a^{i_k},b^{i_k})$ for some $i_k \in \mathbb{N}$ for every $k \in \mathbb{N}$.
(Notice that $i_k \to \infty$ as $k \to \infty$ since $t \neq a^{i_k}$ for any $k \in \mathbb{N}$.)
Now
\begin{align*}
(t^k-t)^{-1} |\tilde{\Gamma}(t^k) - \Gamma(t) - &(t^k-t) \Gamma'(t) |  \\
&\leq (t^k-t)^{-1} |\tilde{\Gamma}(t^k) - \Gamma(a^{i_k}) - (t^k -a^{i_k})\Gamma'(a^{i_k})| \\
& \quad + (t^k-t)^{-1} |(t^k -a^{i_k})\Gamma'(a^{i_k}) - (t^k -a^{i_k})\Gamma'(t)| \\
& \quad + (t^k-t)^{-1} |\Gamma(a^{i_k}) - \Gamma(t) - (a^{i_k}-t) \Gamma'(t)|.
\end{align*}
We may bound the first term on the right as follows:
$$
(t^k-t)^{-1} |\tilde{\Gamma}(t^k) - \Gamma(a^{i_k}) - (t^k -a^{i_k})\Gamma'(a^{i_k})|  \leq (t^k-t)^{-1} \int_{a^{i_k}}^{t^k} \left| \tilde{\Gamma}'(s) - \Gamma'(a^{i_k}) \right| \, ds.
$$
By Claim~\ref{deriv_claim}, this is bounded by $\tilde{\varepsilon}$ for large enough $k$ since $t^k - a^{i_k} < t^k-t$.
The second term can be bounded by $|\Gamma'(a^{i_k})-\Gamma'(t)|$.
Since $\Gamma'$ is continuous on $K$, this may also be made less than $\tilde{\varepsilon}$ for large $k$.
Finally, the third term can be made smaller than $\tilde{\varepsilon}$ since 
$\Gamma \in \mathfrak{C}^1(K,\mathbb{R}^{2n+1})$ and since $(a^{i_k}-t) / (t^k-t) \leq 1$.

Since any decreasing sequence $\{ t^k \}$ in $I$ with $t^k \to t$ 
either has a subsequence entirely contained in $K$ or a subsequence entirely contained in $I \setminus K$,
we have proven the differentiability of $\Gamma$ from the right for any $t \in K$ ($t \neq \max \{ K \}$) with right derivative equal to $\Gamma'(t)$.
By an identical argument involving an increasing sequence $\{ t^k \}$ in $I$ with $t^k \to t$ 
when $t \neq b^i$ and $t \neq \min \{ K \}$, 
we have that $\Gamma$ is differentiable from the left at any $t \in K$ ($t \neq \min \{ K \}$) with left derivative $\Gamma'(t)$.
Thus we may conclude the statement of the claim.

\begin{claim}
$\tilde{\Gamma}$ is $C^1$ on $I$.
\end{claim}
We have already shown that $\tilde{\Gamma}$ is differentiable on $I$ with $\tilde{\Gamma}'|_K = \Gamma'$.
Since $\tilde{\Gamma}$ is $C^1$ on each $(a^i,b^i)$, $\tilde{\Gamma}'$ is continuous on $I \setminus K$.
It remains to show that $\tilde{\Gamma}'$ is continuous on $K$.

Fix $t \in K$.
If $t = a^i$ for some $i \in \mathbb{N}$, we showed in the proof of the previous claim that $\tilde{\gamma}'$ is continuous from the right at $t$.
This gives for any $0 < \delta < b^i-a^i$
\begin{align*}
|\tilde{h}'&(a^i + \delta) - \tilde{h}'(a^i)| \\
&\leq 2 \sum_{j=1}^n \left|(\tilde{f}_j'(a^i + \delta) \tilde{g}_j(a^i + \delta) - \tilde{g}_j'(a^i + \delta) \tilde{f}_j(a^i + \delta)) - (f_j'(a^i) g_j(a^i) - g_j'(a^i) f_j(a^i)) \right|
\end{align*}
which vanishes as $\delta \to 0$, and so $\tilde{\Gamma}'$ is continuous from the right at $a^i$.
A similar argument gives continuity of $\tilde{\Gamma}'$ from the left at $b^i$.

Suppose $t \neq a^i$ for any $i \in \mathbb{N}$ and $t \neq \max \{ K \}$.
Let $\{ t^k \}$ be a decreasing sequence in $K$ with $t^k \to t$.
Then $|\Gamma'(t) - \Gamma'(t+\delta^k)|$ 
may be made arbitrarily small when $k$ is large since $\Gamma'$ is continuous on $K$.
If there is a decreasing sequence $\{ t^k \}$ in $I \setminus K$ with $t^k \to t$, then $t^k \in (a^{i_k},b^{i_k})$ for some $i_k \in \mathbb{N}$ for every $k \in \mathbb{N}$, and so
$$
|\Gamma'(t) - \tilde{\Gamma}'(t^k)| \leq |\Gamma'(t) - \Gamma'(a^{i_k})| +|\Gamma'(a^{i_k}) - \tilde{\Gamma}'(t^k)|
$$
may be made arbitrarily small for large $k$ by Claim~\ref{deriv_claim}.
As above, since any decreasing sequence $\{ t^k \}$ in $I$ with $t^k \to t$ 
either has a subsequence entirely contained in $K$ or a subsequence entirely contained in $I \setminus K$,
we have shown that $\tilde{\Gamma}'$ is continuous from the right at $t$.
A similar argument when $t \neq b^i$ and $t \neq \min \{ K \}$ involving an increasing sequence $\{ t^k \}$ gives continuity of $\tilde{\Gamma}'$ from the left on $K$.
This proves the claim

Extending $\tilde{\Gamma}$ from $I$ to $\mathbb{R}$ in a smooth, horizontal way completes the proof of the theorem.
\end{proof}

We will now see that the Luzin approximation of horizontal curves in $\mathbb{H}^n$ follows from the above result as it does in the classical case.
As mentioned in the introduction, this is a new proof of the result of Speight \cite{gareth}.

\begin{corollary}
\label{corollary}
Suppose $\Gamma=(f_1,g_1,\dots,f_n,g_n,h):[a,b] \to \mathbb{H}^n$ is horizontal. 
Then, for every $\varepsilon > 0$, there is a $C^1$, horizontal curve $\tilde{\Gamma}:\mathbb{R} \to \mathbb{H}^n$ and a compact set $E \subset [a,b]$ with $|[a,b] \setminus E| < \varepsilon$ so that $\tilde{\Gamma}(t) = \Gamma(t)$ and $\tilde{\Gamma}'(t) = \Gamma'(t)$ for every $t \in E$.
\end{corollary}
\begin{proof}
Since $\Gamma$ is horizontal, it is absolutely continuous as a mapping into $\mathbb{R}^{2n+1}$.
Thus it is differentiable almost everywhere in $(a,b)$ and the derivative $\Gamma'$ is $L^1$ on $(a,b)$.
Suppose that $t \in (a,b)$ is a point of differentiability of $\Gamma$ 
and that $t$ is a Lebesgue point of $f_j'$ and $g_j'$ for $j=1, \dots , n$.
Define $\bar{\Gamma} = (\bar{f}_1, \bar{g}_1, \dots, \bar{f}_n, \bar{g}_n, \bar{h}):[a,b] \to \mathbb{H}^n$ so that 
$\bar{\Gamma}(s) = \Gamma(t)^{-1} * \Gamma(s)$.
Since $\bar{\Gamma}(t) = 0$ and $\bar{\Gamma}$ is horizontal, we have for any $\delta > 0$ with $t + \delta \in [a,b]$
\begin{align*}
\frac{|\bar{h}(t+\delta)|}{\delta^2} 
&= \frac{1}{\delta^2} \left| \int_t^{t+ \delta} \bar{h}'(s) \, ds \right| 
\leq \frac{2}{\delta} \sum_{j=1}^n \int_t^{t+\delta} \frac{1}{\delta} \left|\bar{f}_j'(s) \bar{g}_j(s) - \bar{f}_j(s) \bar{g}_j'(s) \right| \, ds \\
&\leq \frac{2}{\delta} \sum_{j=1}^n  \int_t^{t+\delta} \left|\bar{f}_j'(s) \frac{\bar{g}_j(s)}{s-t} - \frac{\bar{f}_j(s)}{s-t} \bar{g}_j'(s) \right| \, ds \\
&= \frac{2}{\delta} \sum_{j=1}^n  \int_t^{t+\delta} \left|f_j'(s) \frac{g_j(s)-g_j(t)}{s-t} - \frac{f_j(s)-f_j(t)}{s-t} g_j'(s) \right| \, ds, \\
\end{align*}
and so $\frac{|\bar{h}(t+\delta)|}{\delta^2} \to 0$ as $\delta \to 0$.
Similarly, $\frac{|\bar{h}(t-\delta)|}{\delta^2} \to 0$ as $\delta \to 0$

Notice that $\bar{h}(s) = h(s) - h(t) - 2\sum_{j=1}^n(f_j(s)g_j(t) - f_j(t)g_j(s))$ for every $s \in (a,b)$.
Thus since almost every point in $(a,b)$ is a point of differentiability of $\Gamma$ 
and a Lebesgue point of each $f_j'$ and $g_j'$, 
we have 
\begin{equation}
\label{cond_cor}
\lim_{s \to t} \frac{\left| h(s) - h(t) - 2\sum_{j=1}^n(f_j(s)g_j(t) - f_j(t)g_j(s)) \right|}{(s-t)^2} = 0
\end{equation}
for almost every $t \in (a,b)$.
Denote by $E_1$ the set of all $t \in (a,b)$ satisfying both \eqref{cond_cor} and $\Gamma'(t) \in H_{\Gamma(t)}\mathbb{H}^n$.
Hence $|[a,b] \setminus E_1| = 0$.

Let $\varepsilon > 0$.
By Luzin's theorem, $\Gamma'$ is continuous on a compact set $E_2 \subset E_1$ with $|E_1 \setminus E_2| < \varepsilon/3$.
By applying Egorov's theorem to the pointwise convergent sequence of functions $\{ \psi_k \}$ defined on $E_2$ as
$$
\psi_k(t) = \sup_{s \in (t-\frac{1}{k},t+\frac{1}{k})} \left\{ \frac{|\Gamma(s)-\Gamma(t)-(s-t)\Gamma'(t)|}{|s-t|} \right\},
$$
we see that $\Gamma \in \mathfrak{C}^1(E_3,\mathbb{R}^{2n+1})$ with Whitney derivative $\Gamma'$ for a compact set 
$E_3 \subset E_2$ with $|E_2 \setminus E_3| < \varepsilon/3$.
Once again applying Egorov's theorem to the convergent sequence of functions $\{ \phi_k \}$ defined on $E_3$ as
$$
\phi_k(t) = \sup_{s \in (t-\frac{1}{k},t+\frac{1}{k})} \left\{ \frac{\left| h(s) - h(t) - 2\sum_{j=1}^n(f_j(s)g_j(t) - f_j(t)g_j(s)) \right|}{(s-t)^2} \right\},
$$
we conclude that \eqref{area} holds on a compact set 
$E_4 \subset E_3$ with $|E_3 \setminus E_4| < \varepsilon/3$.

Thus $\Gamma$ is of Whitney class $\mathfrak{C}^1(E_4,\mathbb{R}^{2n+1})$, and conditions \eqref{area} and \eqref{horiz} hold on the compact set $E_4$.
Therefore, by Theorem~\ref{main}, 
there is a $C^1$, horizontal $\tilde{\Gamma}:\mathbb{R} \to \mathbb{H}^n$ so that $\tilde{\Gamma}(t) = \Gamma(t)$ and $\tilde{\Gamma}'(t) = \Gamma'(t)$ for every $t \in E_4$ where $|[a,b] \setminus E_4| < \varepsilon$.
\end{proof}

\section{Appendix: Proof of the lemma}
\begin{proof}

Fix $i \in \mathbb{N}$ and $j \in \{ 1 ,\dots, n \}$.

To simplify notation, write 
$\alpha = \gamma_j'(a^i) \cdot \mathbf{u}_j^i$, 
$\beta = \gamma_j'(b^i) \cdot \mathbf{u}_j^i$, 
$\mu = \gamma_j'(a^i) \cdot \mathbf{v}_j^i$, 
$\nu = \gamma_j'(b^i) \cdot \mathbf{v}_j^i$, and 
\begin{equation}
\label{lambda}
\lambda = \frac{1}{n} \left[ h(b^i)-h(a^i) - 2\sum_{k=1}^n(f_k(b^i)g_k(a^i) - f_k(a^i)g_k(b^i)) \right],
\end{equation}
and so $|\lambda|/(b^i-a^i)^2 < \varepsilon^i$. 
In other words, $\alpha$ and $\beta$ are the components of the mapping $\gamma_j'$ at $a^i$ and $b^i$ respectively
in the direction of the segment $\overline{\gamma_j(a^i) \gamma_j(b^i)}$, 
and $\mu$ and $\nu$ are its components in the perpendicular direction.

First, we prove that $|\mu|<\varepsilon^i$ and $|\nu|<\varepsilon^i$.
Indeed, the magnitude of $(\gamma_j(b^i)-\gamma_j(a^i)-(b^i-a^i) \gamma_j'(a^i))$ is at least equal to the magnitude of its projection along $\mathbf{v}_j^i$.
That is,
$$
\varepsilon^i > \frac{|\gamma_j(b^i)-\gamma_j(a^i)-(b^i-a^i) \gamma_j'(a^i)|}{b^i-a^i} 
\geq \frac{|(\gamma_j(b^i)-\gamma_j(a^i)-(b^i-a^i) \gamma_j'(a^i)) \cdot \mathbf{v}_j^i|}{b^i-a^i} \\
=|\mu|
$$
since $\gamma_j(b^i)-\gamma_j(a^i)$ is orthogonal to $\mathbf{v}_j^i$.
Replacing $\gamma_j'(a^i)$ with $\gamma_j'(b^i)$ in this argument gives $|\nu|<\varepsilon^i$.
We also have 
\begin{equation}
\label{alpha}
\varepsilon^i > \frac{|(\gamma_j(b^i)-\gamma_j(a^i)-(b^i-a^i) \gamma_j'(a^i)) \cdot \mathbf{u}_j^i|}{b^i-a^i} \\
=\left| \frac{|\gamma_j(b^i)-\gamma_j(a^i)|}{b^i-a^i} - \alpha \right|
\end{equation}
since $(\gamma_j(b^i)-\gamma_j(a^i)) \cdot \mathbf{u}_j^i = |\gamma_j(b^i)-\gamma_j(a^i)|$.
Replacing $\gamma_j'(a^i)$ with $\gamma_j'(b^i)$ in this argument gives 
$\left| \frac{|\gamma_j(b^i)-\gamma_j(a^i)|}{b^i-a^i} - \beta \right| < \varepsilon^i$.

Define $P: [0,\infty) \to [0, \infty)$ as $P(t) =C' (t^{1/2} + t^2)$ where $C'$ is a positive constant whose value will be determined by the constructions of $\eta_j^i$ and will depend only on $M$. 
In particular, the value of $C'$ will not depend on $i$ or $j$.

We will first prove the lemma in the case when $\gamma_j$, $\gamma_j'$, $a^i$, and $b^i$ satisfy
\begin{equation}
\label{case}
\left|\alpha + \beta - 9 \frac{|\gamma_j(b^i)-\gamma_j(a^i)|}{b^i-a^i} \right| > \sqrt{\varepsilon^i}.
\end{equation}

Write $\ell = |\gamma_j(b^i)-\gamma_j(a^i)|$.
By composing with a translation of the real line, 
we may assume without loss of generality that $a^i = 0$ and write $\delta := b^i$.
The bounds before the statement of the lemma give $\delta < \varepsilon^i$, $\ell < \varepsilon^i$, and $\ell/\delta < M$.
Define the curve $\eta_j^i = (x_j^i,y_j^i): [0,\delta] \to \mathbb{R}^2$ as follows:
\begin{align*}
x_j^i(t) &= At^3 + Bt^2 + Ct \\
&= \frac{\delta(\alpha+\beta) - 2 \ell}{\delta^3}t^3 + \frac{-\delta(2\alpha+\beta)+3\ell}{\delta^2}t^2 + \alpha t
\end{align*}
and
\begin{align*}
y_j^i(t) &=Dt^4 + Et^3 + Ft^2 + Gt \\
&= 7\frac{6 \delta \ell(\mu-\nu)+\delta^2(\alpha \nu-\beta \mu)-15\lambda}{2 \delta^4(\delta(\alpha+\beta)-9\ell)}t^4 \\
&+ \frac{\delta \ell(33\nu-51\mu)+\delta^2(\alpha(\mu-6\nu)+\beta(8\mu+\nu))+105\lambda}{\delta^3(\delta(\alpha+\beta)-9\ell)}t^3 \\
&- \frac{\delta \ell(24\nu-78\mu)+\delta^2(4\alpha \mu+11\beta \mu-5\alpha \nu+2\beta \nu)+105\lambda}{2\delta^2(\delta(\alpha+\beta)-9\ell)}t^2 + \mu t.
\end{align*}

One may check (by hand or with Mathematica; I did both) that this curve satisfies conditions (\ref{cond_point}), (\ref{cond_deriv}), and (\ref{cond_area}). 
Now
\begin{align*}
|x_j^i(t)| &\leq |A| \delta^3 + |B| \delta^2 + |C| \delta \\
&\leq \delta(|\alpha|+|\beta|) + 2 \ell + \delta(2|\alpha|+|\beta|)+3\ell + |\alpha| \delta \\
&< \varepsilon^i(M+M)+2\varepsilon^i+\varepsilon^i(2M + M) + 3\varepsilon^i + M\varepsilon^i 
= (6M+5)\varepsilon^i < P(\varepsilon^i)
\end{align*}
after choosing $C'$ large enough 
(since either $\varepsilon^i \leq (\varepsilon^i)^2$ or $\varepsilon^i \leq (\varepsilon^i)^{1/2}$). 
Also,
\begin{align*}
|(x_j^i)'(t) - \alpha| &\leq 3|A| \delta^2 + 2|B| \delta + |C-\alpha| \\
&\leq 3 \left| \alpha+\beta - 2 \frac{\ell}{\delta} \right| + 2 \left| -2\alpha-\beta+3\frac{\ell}{\delta} \right| \\
&< 3\varepsilon^i + 3\varepsilon^i + 4\varepsilon^i + 2\varepsilon^i= 12\varepsilon^i < P(\varepsilon^i)
\end{align*}
for large enough $C'$ since $\left| \alpha-  \frac{\ell}{\delta} \right| < \varepsilon^i$ and $\left| \beta-  \frac{\ell}{\delta} \right| < \varepsilon^i$ by \eqref{alpha}.

Now we will consider the sizes of $y$ and $y'$.
First, we examine the size of terms in $D$. We have
$$
\left| \frac{\delta \ell(\mu-\nu)}{\delta^4(\delta(\alpha+\beta)-9\ell)} \right| \delta^3 
= \frac{\frac{\ell}{\delta}|\mu-\nu|}{|\alpha+\beta-9\frac{\ell}{\delta}|} 
<\frac{2M \varepsilon^i}{\sqrt{\varepsilon^i}} = 2M\sqrt{\varepsilon^i}
$$
by \eqref{case}.
Similarly,
$$
\left| \frac{\delta^2(\alpha \nu-\beta \mu)}{\delta^4(\delta(\alpha+\beta)-9\ell)} \right| \delta^3 
= \frac{|\alpha \nu|+|\beta \mu|}{|\alpha+\beta-9\frac{\ell}{\delta}|} 
<2M\sqrt{\varepsilon^i}.
$$
Finally,
$$
\left| \frac{\lambda}{\delta^4(\delta(\alpha+\beta)-9\ell)} \right| \delta^3 = \frac{|\frac{\lambda}{\delta^2}|}{|\alpha+\beta-9\frac{\ell}{\delta}|} <\sqrt{\varepsilon^i}
$$
by \eqref{lambda}.
Thus $4|D|\delta^3 < P(\varepsilon^i)$ for large enough $C'$. %$4 \cdot \frac{7}{2}(12M\sqrt{\varepsilon^i} + 2M\sqrt{\varepsilon^i}+\frac{60\sqrt{\varepsilon^i}}{4}) = (196M + 210)\sqrt{\varepsilon^i}$.
Identical arguments may be applied to show that $3|E|\delta^2 < P(\varepsilon^i)$ %$(147M+468)\sqrt{\varepsilon^i}$ 
and $2|F|\delta < P(\varepsilon^i)$. %$(248M + 840)\sqrt{\varepsilon^i}$.
Therefore,
$$
|y_j^i(t)| \leq |D| \delta^4 + |E| \delta^3 + |F| \delta^2 + |G|\delta < P(\varepsilon^i), %(222M + 629.5)(\varepsilon^i)^{3/2}
$$
and
$$
|(y_j^i)'(t) - \mu| \leq 4|D| \delta^3 + 3|E| \delta^2 + 2|F| \delta + |G-\mu| < P(\varepsilon^i) %(591M + 1519)\sqrt{\varepsilon^i},
$$
for any $t \in [0,\delta]$ and large enough $C'$.
This proves (\ref{cond_beta}) and completes the proof of the lemma in this case.

We now consider the case when
$$
\left|\alpha + \beta - 9 \frac{|\gamma_j(b^i)-\gamma_j(a^i)|}{b^i-a^i} \right| \leq \sqrt{\varepsilon^i}.
$$
By composing with a translation as before, we can again write $[a^i,b^i] = [0,\delta]$.
We will first find bounds on $\alpha$ and $\beta$ in this case. We have
$$
\left|\alpha +\beta - 2\frac{\ell}{\delta} \right| - 7 \left|\frac{\ell}{\delta} \right| \leq \left|\alpha + \beta - 9 \frac{\ell}{\delta} \right| \leq \sqrt{\varepsilon^i}
$$
and so we have $|\frac{\ell}{\delta}| < \frac{\sqrt{\varepsilon^i}+2\varepsilon^i}{7}$ since $|\alpha-\frac{\ell}{\delta}| < \varepsilon^i$ and $|\beta-\frac{\ell}{\delta}| < \varepsilon^i$.
Thus also 
$$
|\alpha| < \varepsilon^i + \frac{\sqrt{\varepsilon^i}+2\varepsilon^i}{7} = \frac{\sqrt{\varepsilon^i}+9\varepsilon^i}{7} 
\quad \text{ and } \quad
|\beta| < \frac{\sqrt{\varepsilon^i}+9\varepsilon^i}{7}.
$$

We will define $\eta_j^i=(x_j^i,y_j^i):[0,\delta] \to \mathbb{R}^2$ piecewise on its domain as follows:
$$
x_j^i(t) = \left\{ 
\begin{array}{ll}
\frac{9 \delta \alpha - 27 \ell}{\delta^3} t^3 + \frac{-12\delta \alpha + 27 \ell}{2 \delta^2} t^2 + \alpha t & : t \in [0,\frac{\delta}{3}] \\
\frac{9 \delta \beta - 27 \ell}{\delta^3} t^3 + \frac{-42\delta \beta + 135 \ell}{2 \delta^2} t^2 + \frac{16 \delta \beta-54 \ell}{\delta} t -4 \delta \beta + \frac{29}{2}\ell & : t \in [\frac{2\delta}{3},\delta]
\end{array}
\right.
$$
and 
$$
y_j^i(t) = \left\{ 
\begin{array}{ll}
\frac{9 \mu}{\delta^2} t^3 - \frac{6 \mu}{\delta} t^2 + \mu t & : t \in [0,\frac{\delta}{3}] \\
\frac{9 \nu}{\delta^2} t^3 - \frac{21 \nu}{\delta} t^2 + 16 \nu t -4 \delta \nu & : t \in [\frac{2\delta}{3},\delta]
\end{array}
\right.
.
$$
On $(\frac{\delta}{3},\frac{2\delta}{3})$, define $\eta_j^i$ as
\begin{align*}
x_j^i(t) &= R \cos(\tau(t)) - R + \frac{\ell}{2} \\
y_j^i(t) &= R \sin(\tau(t))
\end{align*}
where $R = \frac{1}{2\sqrt{\pi}} | H |^{1/2}$ with
\begin{align*}
H &:= \lambda - 2 \int_0^{\delta/3} (x'y - x y') - 2 \int_{2\delta/3}^{\delta} (x'y - x y') \\
&= \lambda - \frac{\delta \ell (\mu-\nu)}{15},
\end{align*}
and $\tau:(\frac{\delta}{3}, \frac{2\delta}{3}) \to \mathbb{R}$ is defined as 
$$
\tau(t) = \pm \left(-\frac{108 \pi}{\delta^3} t^3 + \frac{162 \pi}{\delta^2}{t^2} - \frac{72\pi}{\delta}t +10 \pi \right)
$$
where we choose $+$ if $H \leq 0$ and $-$ if $H > 0$.
With this definition, $x_j^i$ and $y_j^i$ are $C^1$ on $[0,\delta]$ and conditions (\ref{cond_point}) and (\ref{cond_deriv}) are satisfied.

We can argue as we did in the proof of the previous case to show that %$|x| < (50M+163) \varepsilon^i$, $|x'| < (\frac{355}{7} \sqrt{\varepsilon^i} + \frac{1305}{7} \varepsilon^i)$, $|y| < 50 (\varepsilon^i)^2$, and $|y'| < 85 \varepsilon^i$ 
$$
 |x_j^i(t)| < P(\varepsilon^i), \quad  |(x_j^i)'(t) - \alpha| < P(\varepsilon^i), \quad |y_j^i(t)| < P(\varepsilon^i), \quad |(y_j^i)'(t) - \mu| < P(\varepsilon^i)
$$
for $t \in [0,\frac{\delta}{3}] \cup [\frac{2\delta}{3},\delta]$ with large enough $C'$.
To prove condition (\ref{cond_beta}), 
it remains to find bounds for $x_j^i$, $y_j^i$, and their derivatives on $(\frac{\delta}{3},\frac{2\delta}{3})$.
We have 
$R \leq \frac{1}{2\sqrt{\pi}} \left( \varepsilon^i + \frac{2}{15} (\varepsilon^i)^3 \right)^{1/2} 
\leq \frac{1}{2\sqrt{\pi}} \left( (\varepsilon^i)^{1/2} + \sqrt{\frac{2}{15}} (\varepsilon^i)^{3/2} \right)$, and so, 
for any $t \in (\frac{\delta}{3},\frac{2\delta}{3})$,
$|x_j^i(t)| %\leq \frac{1}{\sqrt{\pi}} (\varepsilon^i)^{1/2} + \sqrt{\frac{2}{15 \pi}} (\varepsilon^i)^{3/2} + \frac{\varepsilon^i}{2}
< P(\varepsilon^i)$
and 
$|y_j^i(t)| %\leq \frac{1}{2\sqrt{\pi}} (\varepsilon^i)^{1/2} + \frac{1}{ \sqrt{30 \pi}} (\varepsilon^i)^{3/2}.
< P(\varepsilon^i)$ for large enough $C'$.

We will now prove bounds for the derivatives.
Notice that on $(\frac{\delta}{3},\frac{2\delta}{3})$ we have $|(x_j^i)'(t)| \leq |R \tau'(t)|$ and $|(y_j^i)'(t)| \leq |R \tau'(t)|$.
Now for any $t \in (\frac{\delta}{3},\frac{2\delta}{3})$
$$
|\tau'(t)| 
\leq 3\left|\frac{108 \pi}{\delta^3}\right| \left(\frac{2\delta}{3} \right)^2 + 2\left| \frac{162 \pi}{\delta^2} \right| \left(\frac{2\delta}{3} \right) + \left|\frac{72\pi}{\delta}\right|
=\frac{432 \pi}{\delta}.
$$
Therefore, $|(x_j^i)'(t)|^2$ and $|(y_j^i)'(t)|^2$ are bounded by
$$
|R \tau'(t)|^2 
%&\leq 46656 \pi \left( \frac{|\lambda|}{\delta^2} + \frac{2 \left|\int_0^{\delta/3} (x'y - x y') \right|}{\delta^2} + \frac{2 \left|\int_{2\delta/3}^{\delta} (x'y - x y') \right|}{\delta^2} \right) \\
%&=46656 \pi \left( \frac{|\lambda|}{\delta^2} + \frac{\delta \ell \mu}{15 \delta^2} + \frac{ \delta \ell \nu}{15 \delta^2} \right)\\
=46656 \pi \left( \frac{|\lambda|}{\delta^2} + \frac{\delta \ell ( |\mu|+|\nu|)}{15\delta^2} \right) \\
= 46656 \pi \left( \varepsilon^i + \frac{2(\varepsilon^i)^{3/2}}{105}+\frac{4(\varepsilon^i)^2}{105} \right).
$$
Choosing large enough $C'$, we have $|(x_j^i)'(t)|< P(\varepsilon^i)$ and $|(y_j^i)'(t)|< P(\varepsilon^i)$.
Since $|\alpha|<\frac{\sqrt{\varepsilon^i}+9\varepsilon^i}{7}$ and $|\mu|<\varepsilon^i$, this proves condition (\ref{cond_beta}).

It remains to prove condition (\ref{cond_area}). We have $2 \int_{\delta/3}^{2\delta/3} ((x_j^i)'y_j^i - x_j^i (y_j^i)') = \mp 4 \pi R^2$ which is negative if $H < 0$ and positive if $H > 0$ (and so $\mp 4 \pi R^2 = H)$.
Thus
\begin{align*}
2 \int_0^{\delta} ((x_j^i)'y_j^i - x_j^i (y_j^i)') \\
&= 2 \int_0^{\delta/3} ((x_j^i)'y_j^i - x_j^i (y_j^i)') \mp 4 \pi (R)^2 + 2 \int_{2\delta/3}^{\delta} ((x_j^i)'y_j^i - x_j^i (y_j^i)') \\
&= \lambda = \frac{1}{n} \left[h(b^i)-h(a^i) - 2 \sum_{m=1}^n (f_m(b^i)g_m(a^i) - f_m(a^i)g_m(b^i)) \right].
\end{align*}
This completes the proof of the lemma.
\end{proof}

\end{document}